\documentclass[11pt,english]{article}

\usepackage[margin = 1.5cm, bottom=15 mm,footskip=7mm,top=14mm]{geometry}

\usepackage{amsthm}
\usepackage{amsmath}
\usepackage{amssymb}
\usepackage{setspace}
\usepackage{mathtools}
\usepackage{graphicx}
\usepackage[hidelinks]{hyperref}
\usepackage{cleveref}
\usepackage{hyperref}
\usepackage{enumitem}
\usepackage{framed}
\usepackage{subcaption}
\usepackage{xspace}

\usepackage[square,sort,comma,numbers]{natbib}
\usepackage{thmtools}
\usepackage{thm-restate}

\usepackage{floatrow}

\theoremstyle{plain}

\newtheorem*{thm*}{Theorem}
\newtheorem{thm}{Theorem}
\Crefname{thm}{Theorem}{Theorems}

\newtheorem*{lem*}{Lemma}
\newtheorem{lem}[thm]{Lemma}
\Crefname{lem}{Lemma}{Lemmas}

\newtheorem*{claim*}{Claim}

\crefname{claim}{Claim}{Claims}
\Crefname{claim}{Claim}{Claims}

\Crefname{prop}{Proposition}{Propositions}

\crefname{cor}{Corollary}{Corollaries}

\crefname{conj}{Conjecture}{Conjectures}

\Crefname{qn}{Question}{Questions}

\Crefname{obs}{Observation}{Observations}

\Crefname{ex}{Example}{Examples}

\theoremstyle{definition}

\Crefname{prob}{Problem}{Problems}

\Crefname{defn}{Definition}{Definitions}

\theoremstyle{remark}
\newtheorem*{rem}{Remark}

\renewenvironment{proof}[1][]{\begin{trivlist}
\item[\hspace{\labelsep}{\bf\noindent Proof#1.\/}] }{\qed\end{trivlist}}

\newcommand{\remove}[1]{}

\title{\vspace{-1 cm}
A Note on Powers of Paths in Tournaments}
\date{}
\author{Nemanja Dragani\'{c}\thanks{Department of Mathematics, ETH Zurich, nemanja.draganic@math.ethz.ch}\and David Munhá Correia\thanks{Department of Mathematics, ETH Zurich, david.munhacanascorreia@math.ethz.ch}
\and Benny Sudakov\thanks{Department of Mathematics, ETH Zurich, benjamin.sudakov@math.ethz.ch}}

\begin{document}
\maketitle

\section{Introduction}
A folklore result in graph theory states that every tournament contains a directed hamiltonian path. It is then natural to ask whether every tournament contains the $k$-th power $P_n^k$ of a directed path of length $n$, i.e., the graph with vertex set $\{0,\ldots,n\}$ and whose directed edges are $(i,j)$ with $i < j \leq i+k$. This was answered negatively for $k \geq 2$ by Yuster \cite{yuster2020paths}, who showed that there exist tournaments on $n$ vertices which don't contain $P_{2n/3}^2$. This raises a question - how long of a $k$-th power of a path can we always find in a tournament on $n$ vertices? 

Let us denote by $\ell(n,k)$ the maximal $l$ such that every tournament on $n$ vertices contains the $k$-th power of a path of length $\ell$. Yuster proved that every tournament on $n$ vertices contains the square of a path of length $n^{0.295}$. In the same paper, an upper bound of $\ell(n,k) \leq kn/2^{k/2}$ is shown. Recently, Gir\~{a}o \cite{girao2020note} improved the lower bound to $\ell(n,k) \geq n^{1-o(1)}$ for fixed $k$, using the regularity method. Finally, Scott and Kor\'{a}ndi \cite{korandi2020powers} gave a linear bound of $\ell(n,k) \geq n/2^{2^{3k}}$ and asked for the optimal constant $c_k$ for which $\ell(n,k) = (1+o(1))c_kn$. We shorten the gap between $n/2^{2^{3k}}$ and $kn/2^{k/2}$, by showing that $c_k\geq \frac{1}{2^{6k + 7}}$, thus proving that the dependence on $k$ is exponential.

\begin{thm}
For all integers $n,k \geq 2$ we have $\ell(n,k) \geq \frac{n}{2^{6k + 7}}$.
\end{thm}

\section{Proof of Theorem 2}

\begin{proof}
Let $T$ be a tournament on $n$ vertices. First, we order the vertices of $T$ as $v_0,v_1 ,\dots, v_{n-1}$ so that we have the maximal amount of edges which are oriented in increasing direction. For all $i \leq j$, define $V[i,j) = \{v_i, \dots, v_{j-1}\}$.

We begin by giving the following lemma.

\begin{lem}
Let $r = 2^{3k}$ and $t = 2^{6k}$. Let also $0 \leq i \leq n - 100t$ and take a subset $U \subseteq V[i,i+t)$ such that $|U| = r$. Then, there exists a subset $X \subseteq U$ with the following properties:
\begin{itemize}
    \item $|X| = k$
    \item $X$ induces a transitive tournament.
    \item There is a set $U' \subseteq V[i',i'+t)$, for some $i' \in \{i+t,i+2t,\ldots,i+99t\}$, such that $|U'| = r$ and $U'$ lies in the common out-neighbourhood of $X$. 
\end{itemize}
\end{lem}

\begin{proof}
First, we note that every vertex $v \in U$ has at least $49t$ out-neighbours in $B = V[i+t,i+100t)$. Indeed, otherwise suppose we change the ordering of the vertices of $T$ by moving $v$ to the end of the interval $B$. The number of edges oriented in increasing order will increase, which contradicts the maximality of the initial ordering.

Secondly, since $|U| \geq 2^{3k-1}$, we know that there is a subset $X' \subseteq U$ of size $3k$ which induces a transitive tournament. This is a standard exercise which we leave to the reader. 

Now, suppose that there is no $X \subseteq X'$ of size $k$ that has at least $99r$ common out-neighbours in $B$. 

For all $y \in B$, define $d(y) = |N^{-}(y) \cap B|$. By the first observation, we know that there are at least $49|X'|t = 147tk$ edges oriented from $X'$ to $B$. Hence, 
$$\sum_{y \in B} d(y) \geq 147tk$$

Further, consider a vertex $y \in B$. There are ${d(y) \choose k}$ $k$-subsets in $X'$ which lie in the in-neighbourhood of $y$. In turn, no $k$-subset of $X'$ can lie in the neighbourhood of at least $99r$ vertices $y \in B$. Thus, we have

$$\sum_{y \in B} {d(y) \choose k} < 99r {3k \choose k}$$

By convexity of the binomial coefficient and Jensen's inequality, the last two inequalities imply
$$t {1.48k \choose k} < r {3k \choose k}$$
which is a contradiction since $r = 2^{3k}$ and $t = 2^{6k}$.

We can then find such a set $X$. Note that $X$ has size $k$ and induces a transitive tournament. Also note by the pigeonhole principle that there is some $i' \in \{i+t,i+2t,\ldots,i+99t\}$ such that $X$ has at least $r$ common out-neighbours in $V[i',i'+t)$. 
\end{proof}

We can now use Lemma 2 to construct the $k$th power of a path of large length in our tournament. We do this recursively as follows: 

Start with the set $U_0 = V[0,r) \subseteq V[0,t)$; by Lemma 2, there is a subset $X_0 \subseteq U_0$ of size $k$ with the properties stated; therefore, there exists some $i_1 \in \{t,2t,\ldots,99t\}$ and a subset $U_1 \subseteq V[i_1,i_1+t)$ of size $r$ which lies in the common out-neighbourhood of $X_0$; now, do the same thing with $U_1$; there is a subset $X_1 \subseteq U_1$ with the properties stated in the lemma; then, there exists some $i_2 \in \{i_1 + t,i_1 + 2t,\ldots,i_1 + 99t\}$ and a subset $U_2 \subseteq V[i_2,i_2+t)$ of size $r$ which lies in the common out-neighbourhood of $X_1$; continue this process until $i_j > n - 100t$; when this is reached, take a subset $X_j \subseteq U_j$ of size $k$ which induces a transitive tournament.

Note the construction before gives us a sequence of pairwise disjoint sets $X_0,X_1, \dots, X_j$, each one inducing a transitive tournament on $k$ vertices and such that for each $i$, every edge in $E[X_i,X_{i+1}]$ is oriented towards $X_{i+1}$. In turn, this gives the $k$th power of a path of length $(j+1)k - 1$. Note that by construction, we have $i_j \leq 99 j t$ and by considering when the process must stop, we also have $i_j > n - 100t$. This implies $j > \frac{n}{100t} - 1$, thus giving a $k$th power of a path of length at least $\frac{n}{2^{6k+7}}$.
\end{proof}

\begin{rem}
The constant in the exponent of our bound can be optimized, but we made no attempt to do this. We can also use the dependent random choice method instead of Lemma 2 in order to construct the sets $X_0,X_1 \dots, X_j$. This gives a slightly better constant.
\end{rem}


\begin{thebibliography}{1}

\bibitem{girao2020note}
{\sc A.~Gir{\~a}o}, {\em A note on long powers of paths in tournaments},
  preprint arXiv:2010.02875,  (2020).

\bibitem{korandi2020powers}
{\sc D.~Kor\'{a}ndi and A.~Scott}, {\em Powers of paths in tournaments},
  preprint arXiv: 2010.05735,  (2020).

\bibitem{yuster2020paths}
{\sc R.~Yuster}, {\em Paths with many shortcuts in tournaments}, Discrete
  Mathematics, 344 (2020), p.~112168.

\end{thebibliography}
\end{document}